\documentclass[10pt]{article}
\usepackage{amsmath,amssymb,amsthm,amscd}
\usepackage{graphicx}
\numberwithin{equation}{section}

\def\bG{{\mathbf G}}

\def\bF{{\mathbf F}}

\def\bF{{\mathbf F}}

\def\bU{{\mathbf U}}
\def\bV{{\mathbf V}}
\def\bW{{\mathbf W}}

\def\bP{{\mathbf P}}

\def\bU{{\mathbf U}}
\def\bN{{\mathbf N}}
\def\bM{{\mathbf M}}

\def\ZZ{{\mathbb Z}}

\def\RR{{\mathbb R}}
\def\CC{{\mathbb C}}

\newtheorem{prop}{Proposition}[section]
\newtheorem{theo}[prop]{Theorem}
\newtheorem{lemm}[prop]{Lemma}

\newtheorem{rema}[prop]{Remark}

\newtheorem{defi}[prop]{Definition}

\def\begeq{\begin{equation}}
\def\endeq{\end{equation}}

\def\bN{{\mathbf N}}

\title{On uniform K-stability of pairs}
\author{Gang Tian\thanks{Supported partially by
NSF and NSFC grants}\\Peking University}

\date{}

\begin{document}

\maketitle

\noindent

{\bf Abstract}: In this paper, we discuss stable pairs, which were first studied by S. Paul, and give a proof for a result I learned from him. As a consequence, we will show that the K-stability implies the CM-stability.

\tableofcontents

\section{Introduction}

Let $\bf {G}$ be the linear group $\bf {SL}(N+1,\mathbb C)$ and $\bf {V}$, $\bf {W}$ be its two rational representations.
\footnote{Our arguments here also work for any classical subgroups of $\bf {GL}(N+1,\mathbb C)$ after some modifications, moreover,
all representations in this paper are finite dimensional and complex.} By the rationality, say for $\bf {V}$, we mean
that for all $\alpha\in \bf {V}^{\vee}$ (dual space) and $v\in \bf {V}\setminus \{0\}$, the {matrix coefficient} $\varphi_{\alpha , v}$
is a {regular function} on $\bf {G}$, that is, $\varphi_{\alpha , v}\in \mathbb C[\bf {G}]$, where

\begin{equation}\label{eq:coefs}
\varphi_{\alpha , v}: \,\bf G\,\mapsto  \,\mathbb C,~~~~ \varphi_{\alpha , v}(\sigma)\,=\,\alpha(\rho(\sigma) v).
\end{equation}

Recall that by an one-parameter subgroup of $\bG$, we mean a homomorphism: $\lambda:\CC^*\mapsto \bG$, where $\CC^*\,=\,\CC\backslash \{0\}$ is the multiplicative group consisting of all non-zero complex numbers. For any such a $\lambda$ and $v\in \bV$, we can associate a weight as the unique integer $w_{\lambda}(v)$ such that there is a non-zero limit $v_0$ in $\bf {V}$:

\begin{equation}\label{eq:weight-2}
\lim_{t\to 0}\,t^{-w_{\lambda}(v)}\lambda(t) (v)\,=\,v_0 \,\not=\,0.
\end{equation}
Let $\bf {T}$ be a maximal algebraic torus of $\bf {G}$ and write $\bf {gl}(N+1,\mathbb C)$ as $\bf {gl}$ for simplicity.
Given such a $\bf {T}$, its character lattice is defined by

\begin{equation}\label{eq:char-lattice}
\bf {M}_{\mathbb Z}\,= \,\mbox{Hom}_{\mathbb Z}(\bf {T},\mathbb C^*) .
\end{equation}
Its dual lattice, denoted by $\bf {N}_{\mathbb Z}$, consists of all one parameter subgroups
contained in $\bf {T}$. More explicitly, for each $\ell\in \bf {N}_{\mathbb Z}$, the corresponding one-parameter subgroup $\lambda_\ell$ is given by
$$ m(\lambda_\ell(t))\,=\,t^{(\ell , m)},~~~~\forall m\in \bf {M}_{\mathbb Z},~~~t\in \bf {T},$$
where $ (\cdot, \cdot)$ is the standard pairing: $\bf {N}_{\mathbb Z}\times\bf {M}_{\mathbb Z}\,\mapsto\, \mathbb Z $.

We have an associated vector space
$$\bf {M}_{\mathbb R}\,=\, \bf {M}_{\mathbb Z}\otimes_{\mathbb Z}\mathbb{R}\,\cong\,\mathbb R^N.$$
Since $\bf {V}$ is rational, it decomposes under the action of $\bf {T}$ into \emph{weight spaces}

\begin{equation}\label{eq:decomp}
\bf {V}\,=\,\bigoplus_{a\in {A}}\bf {V}_{a},~~~{\rm where}~~{\bf {V}}_{a}\,=\,\{v\in \bf {V} \, |\, t\cdot v\,=\,a(t)\, v \ , \, t \in \bf {T}\}.
\end{equation}
Here $A\,=\, \{a \in \bf {M}_{\mathbb Z}\ | \ \bf {V}_{a}\neq 0\}$.

Given any $v\in \bf {V} \setminus \{0\}$, we denote by $A(v)$ the set of all $a\in A$ with $v_{a}\neq 0$,
where $v_a$ is the projection of $v$ into $\bV_{a}$.

The {\bf weight polytope} $\mathcal{N}(v)$ of $v$ is defined to be the convex hull of $A(v)$ in
$\bf {M}_{\mathbb R}$. Since $\bf {V}$ is a rational representation, $ \mathcal{N}(v)$ is a rational polytope.

There is a natural representation $\bf {gl}(N+1,\CC)$, which consists of all $(N+1)\times(N+1)$ matrices,
by left multiplication:
$$\bG\times \bf {gl}(N+1,\mathbb C)\,\mapsto\, \bf {gl}(N+1,\mathbb C): \,(\sigma, B)\,\mapsto\, \sigma B.$$

The weight polytope $\mathcal N(I)$ of the identity matrix $I$ in $\bf {gl}$
is the standard $N$-simplex which contains the origin in $\bf {M}_\mathbb R$. Then we define degree ${\rm deg}(\bf {V})$ of $\bf {V}$ by

\begin{equation}\label{eq:deg-of-V}
{\rm deg}(\bf {V})\,=\,\min\{\,k\in \mathbb Z\,|\,k > 0~{\rm and}~\mathcal N(v) \subset k\, \mathcal N(I)~{\rm for ~all}~0\not= v\in \bf {V}\,\}.
\end{equation}

Clearly, this definition implies
$$ q\,w_{\lambda}(I)\,\le\, w_{\lambda}(v).$$
where $q\,=\,\deg(\bV)$.

Now we can introduce the notion of K-stability due to S. Paul \cite{paul13}.

\begin{defi}\label{defi:p-stable}
Let $v\in \bV\backslash \{0\}$ and $w\in \bW\backslash \{0\}$.

\noindent

$\bullet$ We call $(v,w)$ K-semistable if for any one-parameter subgroup $\lambda$ of $\bG$,
we have $w_{\lambda}(w)\,\le \, w_{\lambda}(v)$.

\noindent

$\bullet$ We call $(v,w)$ K-stable if it is K-semistable and
$w_{\lambda}(w)\,< \, w_{\lambda}(v)$ whenever the one-parameter subgroup $\lambda$ satisfying:
$\deg(\bV) \,w_{\lambda}(I)\,<\, w_{\lambda}(v)$.

\end{defi}

Here is our main theorem.

\begin{theo}\label{th:HMP}
If $(v,w)$ is $K$-stable, then there is an integer $m > 0$ such that for any one-parameter subgroup $\lambda$ of $\bG$,
we have
\begin{equation}\label{eq:tian0}
m\,( w_{\lambda}(v)\,-\, w_{\lambda}(w))\,\ge\, w_{\lambda}(v)\,-\,\deg(\bV) \,w_{\lambda}(I).
\end{equation}

\end{theo}

\begin{rema}\label{rema:HMP}
We may also say that $(v,w)$ is uniformly K-stable if \eqref{eq:tian0} is satisfied for any one-parameter subgroups.

\end{rema}

We denote by $||\cdot||$ a Hermitian norm on both $\bV$ and $\bW$ and define
\begin{equation}\label{eq:KN-func}
p_{v,w}(\sigma)\,=\,\log||\sigma (w)||^2\,-\,\log||\sigma(v)||^2.
\end{equation}
Then we have

\begin{theo}\label{th:properness}
If $(v,w)$ is K-stable, then there is an integer $m>0$ and a uniform constant $C$ such that
\begin{equation}\label{eq:properness}
m\,p_{v,w}(\sigma)\,\ge\,\deg(V)\log ||\sigma||^2- \log ||\sigma (v)||^2-C.
\end{equation}
\end{theo}

The organization of this paper is as follows: In next section, we recall a theorem on K-semistability first given by S. Paul and a proof of this theorem by S. Boucksom, T. Hisamoto and M. Jonsson. In Section 3, we show a connection between K-stability and positivity of certain line bundle. This line bundle is actually the CM-bundle in the case of constant scalar curvature K\"ahler metrics. In Section 4, we prove Theorem \ref{th:HMP} and its consequence, Theorem \ref{th:properness}. In Section 5, we give an application of our main theorem. We show that the K-stability implies the CM-stability. In Appendix A, we give a proof of a variant of Richardson's lemma due to S. Paul.

\vskip 0.1in

\noindent

{\bf Acknowledgment}: I would like to thank Chenyang Xu and Xiaohua Zhu for inspiring conversations as well as useful comments on earlier version of this paper.

\section{Hilbert-Mumford-Paul criterion}

In this section, we discuss a theorem which was first given by S. Paul in \cite{paul12a}.

\begin{theo}\label{th:HMP2}

Let $(v,w)$ be a pair in $(\bV\setminus\{0\})\times (\bW\setminus\{0\})$, then
$p_{v,w}$ is bounded from below on $\bG$ if and only if $(v,w)$ is K-semistable.

\end{theo}

This theorem has an equivalent form. Consider orbits
$\bf {G} [v,w] \,\subset \,\mathbb{P}(\bV\oplus\bW)$ and $\bG [v,0]\,\subset\, \mathbb{P}(\bV\oplus\{0\})\,\subset \,\mathbb{P}(\bV\oplus\bW)$, where
$[v,w]$ (resp,. $[v,0]$) denotes the corresponding point in the projective space $\mathbb{P}(\bV\oplus\bW)$ (resp. $\mathbb{P}(\bV\oplus\{0\})$).
We denote their closures by $\overline{\bG [v,w]}$ and $\overline{\bG [v,0]}$.

It is shown by S. Paul in \cite{paul12a} that
\begin{equation}\label{eq:distance}
p_{v,w}(\sigma)\,=\,\log \tan^2 d(\sigma ([v,w]), \sigma( [v,0])),
\end{equation}
where $d(\cdot,\cdot)$ is the distance function of
the Fubini-Study metric on $\mathbb{P}(\bV\oplus\bW)$. It follows that
$p_{v,w}$ is bounded from below on $\bG$ if and only if there is a constant $c > 0$ such that
$$d(\sigma ([v,w]),\sigma ([v,0]))\,\ge \, c ~~{\rm on}~\bG.$$
It follows that $p_{v,w}$ is bounded from below on $\bG$ if and only if
\begin{equation}\label{eq:stpaul}
\overline{\bG [v,w]}\,\cap\,\overline{\bG [v,0]}\,=\,\emptyset.
\end{equation}
Therefore, we have

\begin{theo}\label{th:stpaul1}
The pair $(v,w)$ is K-semistable if and only if
\eqref{eq:stpaul} holds.

\end{theo}

\begin{rema}\label{rema:hmss}
In \cite{paul12a}, S. Paul called $(v,w)$ \emph{semistable} if \eqref{eq:stpaul} holds. In the case that
$\bV=\CC$ is a trivial representation, we can take $v=1$, then $(1,w)$ is semistable if and only if $0$ is not in the closure of the affine orbit $\bG w$.
In other words, $w$ is semistable in the usual sense of Geometric Invariant Theory. This shows that the K-stability generalizes the notion of stability in classical Geometric Invariant Theory.

\end{rema}

Let us give a proof of Theorem \ref{th:stpaul1} following S. Boucksom, T. Hisamoto and M. Jonsson in \cite{bhj16}.
This proof is based on the following theorem (\cite{bhj16}, Theorem 5.6).

\begin{theo}\label{th:bhj1}
Let $\bU$ be any rational representation of $\bG$. If the (Zariski) closure of the
$\bG$-orbit of a point $x\in  \mathbb{P}(\bU)$ meets a $\bG$-invariant Zariski closed subset $Z \subset  \mathbb{P}(\bU)$, then some $z \in Z \cap {\overline{\bG x}}$ can be reached by an one-parameter subgroup $\lambda$ of $\bG$, i.e. $\lambda(t)( x )$ converge to $z$ as $t$ tends to $0$.
\end{theo}

Note that a given $z\in Z \cap {\overline{\bG x}}$ may not be reachable by any one-parameter subgroup unless the stabilizer of $z$ in $\bG$ is reductive.

Clearly, in order to prove Theorem \ref{th:stpaul1}, we only need to prove that if $(v,w)$ is not K-stable, then two closures of orbits
$\overline{\bG [v,w]}$ and $\overline{\bG [v,0]}$ do not intersect. If it is not true, then $Z\cap \overline{\bG [v,w]}\not= \emptyset$, where $Z=\mathbb{P}(\bV\oplus\{0\})$. Since $Z$ is a closed $\bG$-invariant subset in $\mathbb{P}(\bV\oplus\bW)$, by Theorem \ref{th:bhj1}, we have an one-parameter subgroup
$\lambda$ of $\bG$ such that
$$\lim_{t\to 0} \,\lambda(t)([v,w])\,\in\,\mathbb{P}(\bV\oplus\{0\}).$$
This is equivalent to saying that $w_{\lambda}(w)\,> \, w_{\lambda}(v)$. This contradicts to the K-semistability condition. So Theorem \ref{th:HMP2} is proved.

\section{Positivity vs stability}

In this section, we will interpret the stability in terms of positivity of certain line bundle. This interpretation will be used in proving our main theorems.
We assume that $(v,w)\in \bV\times\bW$ be as before and is K-semistable.

Let $\pi: \tilde{\mathbb{P}}(\bV, \bW)\mapsto \mathbb{P}(\bV\oplus\bW)$ be the blow-up of $\mathbb{P}(\bV\oplus\bW)$ along subvarieties $\mathbb{P}(\bV\oplus \{0\})$ and $\mathbb{P}(\{0\}\oplus\bW)$. Then $\tilde{\mathbb{P}}(\bV, \bW)$ is a smooth projective variety with a natural $\bG$-action and $\pi$ is an isomorphism on the complement of the exceptional loci over $\mathbb{P}(\bV\oplus \{0\})$ and $\mathbb{P}(\{0\}\oplus\bW)$.
Since neither $v$ nor $w$ is zero, the orbit $\bG [v,w]$ lies in the complement of $\mathbb{P}(\bV\oplus \{0\})$ and $\mathbb{P}(\{0\}\oplus\bW)$, so it lifts to an orbit $\pi^{-1}(\bG[v,w])$ in $\tilde{\mathbb{P}}(\bV, \bW)$. Let $\bar \bG$ be a smooth variety compactifying $\bG$. The action of $\bG$ induces a holomorphic map
$$\phi: \bG\,\mapsto \, \tilde{\mathbb{P}}(\bV, \bW),~~~~\phi(\sigma)\,=\,\pi^{-1}(\sigma([v,w])).$$
Since our assumption that the representations of $\bG$ on $\bV$ and $\bW$ are rational, by \eqref{eq:coefs}, $\phi$ is made of polynomials on $\bG$, so there is a blow-up $\tilde \bG$ of $\bar\bG$ along $\bar\bG\backslash \bG$ such that it extends to a holomorphic map
$$\tilde\phi: \tilde\bG\,\mapsto \, \tilde{\mathbb{P}}(\bV, \bW).$$

There are two natural projections:
$$\pi_\bV: \tilde{\mathbb{P}}(\bV,\bW)\,\mapsto\,\mathbb{P}(\bV\oplus\{0\})~~{\rm and}~~\pi_\bW:\tilde{\mathbb{P}}(\bV,\bW) \,\mapsto\, {\mathbb{P}}(\{0\}\oplus\bW).$$

For any vector space $\bU$, we will denote by $H_\bU$ the hyperplane bundle over $\mathbb{P}(\bU)$. Then its inverse $H_\bU^{-1}$ is the universal bundle over
$\mathbb{P}(\bU)$, so we have
\begin{equation}\label{eq:univ-b}
H_{\bU}^{-1}\backslash Z_\bU\,=\,\bU\backslash \{0\},
\end{equation}
where $Z_\bU$ denotes the zero section of $H_\bU^{-1}$. It follows that for any non-zero $u\in \bU$, we can regard $\sigma(u)$ as a point in the fiber of $H_\bU^{-1}$ over $\sigma ([u])\subset \mathbb{P}(\bU)$. Since $\sigma(u)\not= 0$, $\sigma(u)^{-1}$ can be naturally regarded as a point in the fiber of $H_\bU$ over $\sigma([u])$.
Now we define a line bundle over $\tilde{\mathbb{P}}(\bV,\bW)$:
\begin{equation}\label{eq:lb1}
\bf {L} \,=\, \pi^*_{\bf V} H_{\bf V}^{-1}\otimes \pi_{\bf W}^*H_{\bf W}.
\end{equation}

Put
$$\tilde\pi_{\bf V}=\pi_{\bf V}\cdot \tilde\phi: \tilde{\bf {G}}\,\mapsto\,\mathbb{P}(\bf V\oplus\{0\})~~{\rm and}~~
\tilde\pi_{\bf {W}}=\pi_{\bf W}\cdot\tilde \phi:\tilde{\bf G} \,\mapsto\, {\mathbb{P}}(\{0\}\oplus\bf W).$$
Since they maps $\bf {G}$ onto the orbits $\bG[v]$ and $\bf G[w]$ respectively, we have a natural section $S_{v,w}$ of $\tilde \phi^*\bf L$ over $\bf G$:
$$ S_{v,w}(\sigma)\,=\,\tilde\pi_\bV^*\sigma(v)\otimes \tilde\pi_{\bf W}^*\sigma(w)^{-1},~~~\forall \sigma\in \bf G.$$
The Hermitian norms on $\bf V$ and $\bf W$ induce Hermitian metrics on line bundles $H_{\bf V}$ and $H_{\bf W}$, consequently, we have a natural Hermitian metric $||\cdot||_{\bf L}$ on $\bf L$ and consequently, $\phi^*{\bf L}$.
We observe
$$p_{v,w}(\sigma)\,=\,-\,\log ||S_{v,w}||^2_{\bf L}.$$
Then on $X$, we have
\begin{equation}\label{eq:p=S}
p_{v,w}\,\ge\,- 2 c~~~{\rm if~and~only~of}~~~||S_{v,w}||_{\bf L} \,\le\,e^c.
\end{equation}
Hence, by well-known extension theorem for holomorphic functions,
$S_{v,w}$ can be extended to be a holomorphic section over $\tilde\bG$. The converse is also true. Hence, we have shown

\begin{lemm}\label{lemm:tian-1}
The functional $p_{v,w}$ is bounded from below if and only if $S_{v,w}$ extends as a holomorphic section over
$\tilde \bf G$, equivalently, $(v,w)$ is K-semistable if and only if $S_{v,w}$ extends as a holomorphic section.
\end{lemm}

Next we discuss the case of K-stability. Recall that $\bG$ has a natural representation $\bf {gl}\,=\,\bf {gl}(N+1,\mathbb C)$. Set $\bf U = \bf {gl}^{\otimes q}$ and $u=I^{\otimes q}$,
where $q\,=\, {\rm deg}(\bf V)$. We have said that $(v,u)$ is K-semistable, moreover, for some uniform $c > 0$, we have
$$p_{v,u}(\sigma)\,=\, {\rm deg}(\bf V)\, \log ||\sigma || \,- \,\log ||\sigma (v)||\,\ge \,-\, c,~~~\forall \sigma\in \bf {G}.$$
Here $||\sigma ||$ is actually the Hilbert-Schmidt norm of $\sigma  \in \bf {gl}$.

As above, we have a blow-up variety $\tilde{\mathbb{P}}(\bf V, \bf U)$ of $\mathbb{P}(\bf V\oplus\bf U)$ along subvarieties
$\mathbb{P}(\bf V\oplus\{0\})$ and $\mathbb{P}(\{0\}\oplus\bf U)$, together with a holomorphic map
$$\phi': \bf G\,\mapsto \, \tilde{\mathbb{P}}(\bf V, \bf U),~~~\phi'(\sigma)= \sigma([v,u]).$$
Moreover, we may choose the blow-up $\tilde{ \bf G}$ of $\bar\bf G$ above such that $\psi$ extends to a holomorphic map
$$\tilde \phi': \tilde {\bf G}\,\mapsto\, \tilde{\mathbb{P}}(\bf V, \bf U).$$
As above, we have a line bundle on $\tilde{\mathbb{P}}(\bf V, \bf U)$
\begin{equation}\label{eq:lb2}
\bf L' \,=\, \pi_{\bf V}^* H_{\bf V}^{-1}\otimes \pi_{\bf U}^*H_{\bf U}.
\end{equation}
The orbit $\bf G([v,w],[v,u])$ can be also lifted to a $\bf G$-orbit in
$\tilde{\mathbb{P}}(\bf V, \bf W)\times \tilde{\mathbb{P}}(\bf V, \bf U)$ and induces a holomorphic map:
$$\tilde \psi\,=\,(\tilde\phi,\tilde\phi'): \tilde{\bf G} \,\mapsto\, \tilde{\mathbb{P}}(\bf V, \bf W)\times \tilde{\mathbb{P}}(\bf V, \bf U).$$
We also have two holomorphic sections
$S_{v,w}$ and $S_{v,u}$ of line bundles $\phi^*\bf L$ and $(\phi')^*\bf L'$, which are both equipped with natural Hermitian metrics $||\cdot||_{\bf L}$
and $||\cdot||_{\bf L'}$, over $\tilde{\bf G}$.
Theorem \ref{th:properness} can be put in an equivalent form:

\begin{theo}\label{th:properness2}
If $(v,w)$ is K-stable, then there is an integer $m >0$ such that for some constant $C>0$,
\begin{equation}\label{eq:tian1}
||S_{v,w}||_{\bf L}^{m}\,\le\, C\,||S_{v,u}||_{\bf L'}.
\end{equation}

\end{theo}

Because the vanishing order of $S_{v,u}$ along each irreducible divisor $D\subset \tilde \bG\backslash \bG$ is always finite and there are
only finitely many of irreducible divisors
$\tilde \bG\backslash \bG$, \eqref{eq:tian1} is the same as saying: For any irreducible divisor $D\subset \tilde \bG\backslash \bG$, $S_{v,w} $ vanishes along $D$ so long as $S_{v,u}$ does. This last statement is plausible, but a direct proof is not available. Instead, we will prove Theorem \ref{th:properness} or \ref{th:properness2} through consideration of maximal tori in $\bG$.

\section{Proof of main theorem}

In this section, we prove Theorem \ref{th:HMP} and Theorem \ref{th:properness}.
We will adopt the notations in last section. Let $\bf T$ be a maximal torus of $\bG$, then any other maximal torus $\bf T'$ is of the form
$\bf T'=\tau\cdot \bf T\cdot \tau ^{-1}$.
We will fix a maximal torus $\bf T$. It admits a compactification $\bar \bf T$ which is simply a product of $\mathbb C P^1$ in the case of $\bf G=\bf {SL}(N+1,\mathbb C)$.
Put
$$\bf P\,=\, \tilde{\mathbb{P}}(\bf V, \bf W)\times \tilde{\mathbb{P}}(\bf V, \bf U).$$
Consider the holomorphic map induced by the $\bf G$-action:
$$f: \bf T\times \bf P\,\mapsto \,\bf P,~~~f(\sigma, x)\,=\,\sigma(x).$$
By \eqref{eq:coefs}, this map can be made of polynomials on $\bf G=\bf {SL}(N+1,\mathbb C)$, so there is a blow-up variety
$\pi:\tilde Z \mapsto Z$, where $Z = \bar \bf T\times \bf P$, such that it extends to a holomorphic map
$$\tilde f: \tilde Z \,\mapsto\,\bf P.$$
For each $x\in \bf P$, its preimage in $\bf T\times\bf P$ is given by
$$f^{-1}(x)\,=\,\{(\sigma^{-1}, \sigma(x))~|~ \sigma\in \bf T\}\subset \bf T\times \bf P.$$
This is isomorphic to $\bf T$. We denote by $\tilde \bf T_x$ the closure of $f^{-1}(x)$ in $\tilde Z$, clearly, we have $\tilde{ \bf T}_x\subset \tilde f^{-1}(x)$ which is
closed in $\tilde Z$.

\begin{lemm}\label{lemm:bound}
Let $\omega$ be a fixed K\"ahler metric on $\tilde Z$, then there is a constant $C=C(\omega, \bf T)$, which is independent of $x\in \bf P$, satisfying:
$$ \int_{\tilde \bf T_x} \,\omega^N\,\le\,C,$$
where $N$ is the dimension of $\bf T$.
\end{lemm}

\begin{proof}
There is a subvariety $B\subset \bP$ of complex codimension at least $2$ such that for any $x$ outside $B$,
$\tilde f^{-1}(x)$ is a subvariety of dimension $N$ and consequently, it coincides with $\tilde {\bf T}_x$. Clearly, all $\tilde {\bf T}_x$ for $x\in \bf P\backslash B$ are  homologous to each other, so we have
\begin{equation}\label{eq:integral-1}
\int_{\tilde { \bf T}_x} \,\omega^N\,=\,C(\omega,\bf T), ~~~\forall x\in \bf P\backslash B,
\end{equation}
i.e., these integrals stay as a constant.
For any $x\in B$, we can choose a sequence of $x_i\in \bP\backslash B$ converging to $x$ such that
$\lim x_i = x$. By \eqref{eq:integral-1}, $\tilde{\bf T}_{x_i}$ form a bounded family in $\bf P$, so by taking a subsequence if necessary, we may assume that
$\tilde{ \bf T}_{x_i}$ converge to a subvariety (possibly non-reduced) $D_\infty$ in $\tilde Z$. Clearly, $\tilde{\bf T}_{x}$ is contained in $D_\infty$ as one component. Hence, by using \eqref{eq:integral-1} for $x_i$, we get
\begin{equation}\label{eq:integral-2}
\int_{\tilde {\bf T}_x} \,\omega^N\,\le \,C.
\end{equation}
The lemma is proved.

\end{proof}

Next we recall $\bf L$ over $\tilde{\mathbb{P}}(\bf V, \bf W)$ and $\bf L'$ over $\tilde{\mathbb{P}}(\bf V, \bf U)$ introduced in last section. Through projections, they pull back to two line bundles $\tilde {\bf L}$ and $\tilde {\bf L}'$ over $\tilde Z$. For any $\tau\in \bf G$, we put $x=(\tau([v,w]),\tau([v,u]))$, then restrictions of $\tilde \bf L$ and $\tilde \bf L'$ to each $\tilde {\bf T}_x\subset \tilde Z$ have two holomorphic sections $S_{v,w}$ and $S_{v,u}$ constructed in a similar way as we did in last section. Actually, if we identify $\tilde T_x$ with a closure of the orbit $\bf T\cdot\tau \subset \bf G$ in $\tilde {\bf G}$, then these are simply those $S_{v,w}$ and $S_{v,u}$ over $\tilde{\bf G}$ from last section restricted to $\tilde {\bf T}_x$. Note that both $\tilde {\bf L}$ and $\tilde {\bf L}'$ have naturally induced Hermitian metrics $||\cdot||_{\tilde{\bf L}}$ and $||\cdot||_{\tilde{\bf L}'}$.

\begin{prop}\label{prop:tian1}
There is an uniform integer $m>0$ such that for $\tau\in \bG$, there is a constant $C_x$,
\begin{equation}\label{eq:tian3}
||S_{v,w}||_{\tilde{\bf L}}^{m}\,\le\, C_x\,||S_{v,u}||_{\tilde{\bf L}'}~~~{\rm on}~~\bar{ \bf T}_x.
\end{equation}

\end{prop}

\begin{proof}
We will prove \eqref{eq:tian3} in two steps. Without loss of generality, we may take $\tau =Id$.
First we claim that for any irreducible divisor $D\subset \tilde {\bf T}_x\backslash \bf T(x)$, $S_{v,w}$ vanishes along $D$ whenever $S_{v,u}=0$ on $D$. This can be proved by applying a variant of Richardson's lemma which already appeared in \cite{paul13}.

\begin{lemm}\label{lemm:richardson}
For any $z\in \tilde{ \bf T}_x\backslash \bf T (x)$, there is a one-parameter subgroup $\lambda: \mathbb C^* \mapsto \bf T$ such that $\lambda(t)( x) \rightarrow \varsigma(z)$ for some $\varsigma\in \bf T$ as $t \rightarrow 0$.
\end{lemm}

Now we choose $z$ to be a generic point in $D$, by the above lemma, we have a one-parameter subgroup $\lambda$ satisfying:
$$\lim_{t\to 0} \lambda(t)( x)\, = \,\varsigma (z).$$
Since $D$ is $\bf T$-invariant, $\varsigma(z)$ is a generic point in $D$. By the K-stability condition, $ S_{v,w}$ restricted to the closure of $\lambda(\mathbb C^*)$ must vanish at $\varsigma(z)$, so $ S_{v,w}$ vanishes along $D$.

Next we show that there is an uniform upper bound on vanishing order of $S_{v,u}$ along any irreducible components of $\tilde {\bf T}_x\backslash {\bf T}( x)$. To see this, we take the K\"ahler metric $\omega$ on $\tilde Z$ as in Lemma \ref{lemm:bound} and denote by $R_{\tilde\bf L'}$ the curvature form of the Hermitian metric $||\cdot||_{\tilde\bf L'}$.
Then we have
\begin{equation}\label{eq:tian4}
\sum_{i} m_i\,[D_i]\,-\,R_{\tilde{\bf L}'} \,=\,\sqrt{-1} \,\partial\bar\partial \log || S_{v,u}||^2_{\tilde{\bf L}'},
\end{equation}
where $\{D_i\}_i$ is the collection of irreducible components of $\tilde {\bf T}_x\backslash \bf T(x)$ and each $m_i$ is the vanishing order of $S_{v,u}$ along $D_i$.
Integrating $\omega^{N-1}$ on both sides of \eqref{eq:tian4}, we get
\begin{equation}\label{eq:tian5}
\sum_{i} m_i\,\int_{D_i} \,\omega^{N-1}\,=\,\int_{\tilde{ \bf T}_x} \,R_{\tilde{\bf L}'}\wedge \omega^{N-1}.
\end{equation}
Since the curvature $R_{\tilde {\bf L}'}$ is bounded by a multiple of $\omega$, by Lemma \ref{lemm:bound},
the right side of \eqref{eq:tian5} is bounded from above by an uniform constant. Hence,
if we let $m$ be a positive integer bigger than the right side of \eqref{eq:tian5}, we have
$m_i\le m$ for each $i$. We have seen above that $S_{v,w}$ vanishes along $D_i$ so long as $m_i\ge 1$, so $ S_{v,w}^m \cdot S_{v,u}^{-1} $ extends to a holomorphic section of $\tilde{\bf L}\otimes (\tilde{\bf L}')^{-1}$ over $\tilde{\bf T}_x$. Thus \eqref{eq:tian3} holds and our proposition is proved.

\end{proof}

Now let $m$ be given in Proposition \ref{prop:tian1}, put $x=([v,w],[v,u])$ and consider a function $F_x$ on $\bG$:
\begin{equation}\label{eq:tian7}
F_x(\sigma)\,=\,m(\log ||\sigma(w)||^2 - \log ||\sigma(v)||^2)\,-\,(\log ||\sigma(u)||^2-\log ||\sigma(v)||^2).
\end{equation}
It is the same as
\begin{equation}\label{eq:tian8}
F_x(\sigma)\,=\,-\,\log \left(\frac{||S_{v,w}||_{\bf L}^{2m}}{||S_{v,u}||_{\bf L'}^2}\right).
\end{equation}
For each one-parameter subgroup $\lambda\subset \bf G$, we have
\begin{equation}\label{eq:tian9}
F_x(\lambda(t))\,=\,[m(w_\lambda(w)- w_{\lambda}(v)) - (q \,w_{\lambda}(I)- w_{\lambda}(v))]\,\log|t|^2\,+\,{\mathcal O}(1),
\end{equation}
where ${\mathcal O}(1)$ denotes a bounded quantity.

Each one-parameter subgroup $\lambda$ is contained in a maximal torus $\tau\cdot \bf T\cdot\tau^{-1}$
of $\bf G$. Then $\bar\lambda = \tau^{-1} \cdot \lambda\cdot \tau$ is an one-parameter subgroup in $\bf T$. Note that
$$F_x(\lambda(t))\,=\,F_x(\tau\cdot\bar\lambda\cdot \tau^{-1}).$$
Using this and \eqref{eq:tian7}, we can deduce the following expansion as we did in \eqref{eq:tian9},
\begin{equation}\label{eq:tian9-1}
F_x(\lambda(t))\,=\,[m(w_{\bar\lambda}(\bar w)- w_{\bar\lambda}(\bar v)) - (q w_{\bar\lambda}(\tau)- w_{\bar\lambda}(\bar v))]\log|t|^2+{\mathcal O}(1),
\end{equation}
where $\bar w = \tau^{-1}(w)$ and $\bar v=\tau^{-1}(v)$.
This, together with \eqref{eq:tian9}, implies that the coefficients in front of $\log |t|^2$ in \eqref{eq:tian9} and \eqref{eq:tian9-1} are equal. In fact, by similar arguments using expansions, we have
\begin{eqnarray*}
w_\lambda(w)- w_{\lambda}(v) &=& w_{\bar\lambda}(\bar w)- w_{\bar\lambda}(\bar v))\\
q \,w_{\lambda}(I)- w_{\lambda}(v)&=& q \,w_{\bar\lambda}(\iota)- w_{\bar\lambda}(\bar v),
\end{eqnarray*}
where $\iota\,=\,\tau^{-1}$. We also have
\begin{equation}\label{eq:tian10}
F_{\iota(x)}(\bar\lambda(t))\,=\,[m(w_{\bar\lambda}(\bar w)- w_{\bar\lambda}(\bar v)) - (q \,w_{\bar\lambda}(\iota)- w_{\bar\lambda}(\bar v))]\, \log|t|^2
+{\mathcal O}(1).
\end{equation}

By \eqref{eq:tian8} and Proposition \ref{prop:tian1}, $F_{\iota(x)}(\bar\lambda(t))=F_x(\bar\lambda(t)\cdot\iota)$ is bounded from below on $\bf T$. It follows from this and
\eqref{eq:tian10} that
$$m(w_{\bar \lambda}(\bar w)- w_{\bar \lambda}(\bar v)) - (q \,w_{\bar \lambda}(\iota)- w_{\bar \lambda}(\bar v))\,\le \,0.
$$
Consequently, we have
\begin{equation}\label{eq:tian6}
m(w_\lambda(w)- w_{\lambda}(v)) - (q \,w_{\lambda}(I)- w_{\lambda}(v))\,\le \,0.
\end{equation}
This concludes the proof of Theorem \ref{th:HMP}.

Next, by using \eqref{eq:tian6} and Theorem \ref{th:HMP2} (also \cite{bhj16}, Theorem 5.4 (ii)), we conclude the proof of Theorem \ref{th:properness}

\section{K-stability vs CM-stability}

In this section, as an application of Theorem \ref{th:HMP}, we give a proof of the main theorem in \cite{tian14}.\footnote{The proof in \cite{tian14} needs additional arguments for the reduction from $\bG$ to it maximal tori. It was thought that one can complete this reduction by applying a result in \cite{paul12a}. It turns out that additional arguments are needed and provided here.} We will follow closely discussions in \cite{tian14}.

Let $M$ be a projective manifold polarized by an ample line bundle $L$.
By the Kodaira embedding theorem, for $\ell$
sufficiently large, a basis of $H^0(M, L^\ell)$ gives an embedding
$\phi_\ell: M \mapsto \CC P^{N}$, where $N=\dim_\CC
H^0(M,L^\ell)-1$. Any other basis gives an embedding of the form
$\sigma\cdot \phi_\ell$, where $\sigma\in \bf G=\bf {SL}(N+1, \CC)$. We fix such an embedding.

The CM-stability introduced in \cite{tian97} is defined in terms of Mabuchi's K-energy:
\begin{equation}
\label{eq:cm-1}
\bM_{\omega_0}(\varphi)\,=\,-\frac{1}{V}\,\int_0^1 \int_M\, \varphi \,({\rm Ric}(\omega_{t\varphi}) - \mu \,\omega_{t\varphi} )\wedge \omega_{t\varphi}^{n-1} \wedge dt,
\end{equation}
where $\omega_0$ is a K\"aher metric with K\"ahler class $2\pi c_1(L)$ and
\begin{equation}\label{eq:mu}
\omega_\varphi\,= \,\omega_0 \,+\,\sqrt{-1}\,\partial\bar\partial \,\varphi\,
~~{\rm and}~~ \,\mu\,=\,\frac{c_1(M)\cdot c_1(L)^{n-1}}{c_1(L)^n}.
\end{equation}
Given an embedding $M\subset \mathbb C P^N$ by $K_M^{-\ell}$, we have an induced function on $\bf G\,=\,\bf {SL}(N+1,\CC)$ which acts on $\mathbb C P^N$:
\begin{equation}\label{eq:cm-2}
\bf F(\sigma)\,=\,\bM_{\omega_0}(\psi_\sigma),
\end{equation}
where $\psi_\sigma$ is defined by
\begin{equation}\label{eq:cm-3}
\frac{1}{\ell}\,\sigma^*\omega_{FS}\,=\,\omega_0\,+\, \sqrt{-1}\,\partial\bar\partial\, \psi_\sigma.
\end{equation}
Note that $\bF(\sigma)$ is well-defined since $\psi_\sigma$ is unique modulo addition of constants. Similarly, we can define
$\bf J$ on $\bG$ by
\begin{equation}\label{eq:cm-3'}
\bf J(\sigma)\, =\, \bf J_{\omega_0}(\psi_\sigma),
\end{equation}
where
\begin{equation}
\label{eq:cm-3''}
J_{\omega_0}(\varphi)\,=\, \sum_{i=0}^{n-1} \frac{i+1}{n+1} \int_M
\sqrt{-1} \, \partial \varphi\wedge \bar{\partial} \varphi \wedge
\omega_0^i\wedge \omega_\varphi^{n-i-1}.
\end{equation}

\begin{defi}\label{defi:cm-1}
We call $M$ CM-semistable with respect to $L^{\ell}$ if $\bF$ is bounded from below and
CM-stable with respect to $L^{\ell}$ if $\bF$ bounded from below and is proper modulo $\bf J$, i.e.,
for any sequence $\sigma_i\in \bG$,
\begin{equation}\label{eq:cm-defi}
\bf F(\sigma_i) \rightarrow\infty~{\rm whenever}~\bf J(\sigma_i) \rightarrow \infty.
\end{equation}
We say $(M,L)$ CM-stable (resp. CM-semistable) if $M$ is CM-stable (resp. CM-semistable) with respect to $L^{\ell}$ for all sufficiently large $\ell$.
\end{defi}

\begin{theo}\label{th:k-cm-1}
Let $(M,L)$ be a polarized projective manifold which is K-stable. Then $M$ is CM-stable with respect to any $L^\ell$ which is
very ample. In particular, $(M,L)$ is CM-stable.
\end{theo}

We refer the readers to \cite{tian13-2} for the definition of the K-stability. Clearly, Theorem \ref{th:k-cm-1} follows from
the following theorem.

\begin{theo}\label{th:k-cm-2}
Let $(M,L)$ be a polarized projective manifold which is K-stable with respect to $L^\ell$. Then there are positive constants $\delta$ and $C$, which may depend on $\ell$, such that
\begin{equation}\label{eq:cm-defi-uni}
\bf F(\sigma) \,\ge\,\delta\,\bf J(\sigma)\,-\,C~~~{\rm on}~\bf G.
\end{equation}
\end{theo}

The rest of this section is devoted to the proof of Theorem \ref{th:k-cm-2}.

First we recall Theorem 2.4 in \cite{tian14} which relates the K-stability to the asymptotic behavior of the K-energy.

\begin{theo}\label{th:k-2}
If $(M, L)$ is K-stable with respect to $L^\ell$, then
$\bf F$ is proper along any one-parameter subgroup $\lambda$ of $\bf G$.
\end{theo}

Here by properness along $\lambda$, we mean that $\bf F$ is bounded from below along $\lambda$ and for any sequence $t_i\to 0$, $\bf F(\lambda(t_i))$ diverge to $\infty$ whenever $\bf J(\lambda(t_i))\to \infty$.

Next we recall the Chow coordinate and Hyperdiscriminant of $M$ (\cite{paul08}):
Let $G(N-n-1,N)$ the Grassmannian of all $(N-n-1)$-dimensional subspaces in
$\mathbb C P^N$. We define
\begin{equation}\label{eq:chow-1}
Z_M\,=\,\{\, P \in G(N-n-1,N)\,|\,P\cap M\,\not=\,\emptyset\,\}.
\end{equation}
Then $Z_M$ is an irreducible divisor of $G(N-n-1,N)$ and determines a non-zero homogeneous polynomial $R_M \in \CC[M_{(n+1)\times (N+1)}] $,
unique modulo scaling, of degree $(n+1) d$, where $M_{k\times l}$ denotes the space of all $k\times l$ matrices.
We call $R_M$ the Chow coordinate or the $M$-resultant of $M$.

Next consider the Segre embedding:
$$M\times \CC P^{n-1} \subset \CC P^N\times \CC P^{n-1} \mapsto \mathbb{P}(M_{n\times (N+1)}^\vee),$$
where $M_{k\times l}^\vee$ denotes its dual space of $M_{k\times l}$. Then we define
\begin{equation}\label{eq:chow-2}
Y_M\,=\,\{\, H\, \subset\,\mathbb{P}(M_{n\times (N+1)}^\vee)\,|\, T_p(M\times \CC P^{n-1}) \,\subset\, H~{\rm for~some}~p\,\}.
\end{equation}
Then $Y_M$ is a divisor in $\mathbb{P}(M_{n\times (N+1)}^\vee)$ of degree $\bar d\,=\,(n (n+1) - \mu )\, d$. This determines
a homogeneous polynomial $\Delta_M$ in $\CC[M_{n\times (N+1)}]$, unique modulo scaling, of degree $\bar d$. We call $\Delta_M$ the hyperdiscriminant of $M$.

Set
$$r = (n+1) \,d \bar d, ~~\bV\,=\,C_r [M_{(n+1)\times (N+1)}],~~\bW\,=\,C_r [M_{n \times (N+1)}],$$
where $C_r[\CC^k]$ denotes the space of homogeneous polynomials of degree $r$ on $\CC^k$.
Following \cite{paul12b}, we associate $M$ with the pair $(R(M), \Delta(M))$ in $\bV\times \bW$, where
$R(M)=R_M^{\bar d}$ and $\Delta(M)=\Delta_M^{(n+1)d}$.
It follows from \cite{paul08} that
\begin{equation}\label{eq:cm-8}
|\,\bF(\sigma ) \,-\, a_n\,p_{R(M),\Delta(M)}(\sigma) \,|\, \leq\,C,~~~\forall \sigma\in\bG
\end{equation}
where $a_n > 0 $ and $C$ are uniform constants.

For each one-parameter subgroup $\lambda\in \bN_\ZZ$, we have
\begin{equation}\label{eq:cm-8'}
p_{R(M),\Delta(M)}(\lambda(t))\,=\,\left(w_\lambda(\Delta(M))\,-\,w_\lambda(R(M))\right)\,\log |t|^2\,+\,{\mathcal O}(1).
\end{equation}
Since $\bF$ is bounded from below on $\bG$, we deduce from \eqref{eq:cm-8} and \eqref{eq:cm-8'} that
\begin{equation}\label{eq:cm-9}
w_\lambda(R(M))\,-\,w_\lambda(\Delta(M))\,\ge\,0.
\end{equation}
Hence, $(R(M),\Delta(M))$ is K-semistable as a pair. Moreover, we see
\begin{equation} \label{eq:equ-1}
\lim_{t\to 0} \,\bF(\lambda(t) )\,=\,\infty~~\Longleftrightarrow~~w_\lambda(R(M))\,-\,w_\lambda(\Delta(M))\,>\,0.
\end{equation}

On the other hand, by Lemma 3.2 in \cite{tian14}, we have
\begin{equation}\label{eq:cm-7}
|\,\bf J(\sigma) \,-\, p_{R(M), I^r} (\sigma)\,| \,\le\, C,~~~\forall \sigma\in\bf G.
\end{equation}
Here, as in previous sections, $I$ denotes the identity in $\bf {gl}$ and $I^r\in \bf U=\bf {gl}^{\otimes r}$.

For any one-parameter subgroup $\lambda$, we also have
\begin{equation}\label{eq:cm-7'}
p_{R(M),I^r}(\lambda(t))\,=\,\left(r\,w_\lambda(I)\,-\,w_\lambda(R(M))\right)\,\log |t|^2\,+\,{\mathcal O}(1).
\end{equation}
Combining this with \eqref{eq:cm-7}, we have
\begin{equation}\label{eq:equ-2}
\lim_{t\to 0} \,\bf J(\lambda(t))\,=\,\infty~~\Longleftrightarrow ~~w_\lambda(R(M))\,-\,r\,w_\lambda(I)\,>\,0.
\end{equation}
Using \eqref{eq:equ-1}, \eqref{eq:equ-2} and Theorem \ref{th:k-2}, we show that $(R(M),\Delta(M))$ is K-stable as a pair.

Hence, Theorem \ref{th:k-cm-2} follows from Theorem \ref{th:properness}.

\section{A final remark}

In this section, we discuss another approach to proving Theorem \ref{th:HMP}. I strongly believe that this approach can be worked out. If so, the resulting proof would be simpler than the one we gave in Section 4.
We will adopt notations from Section 1 and 2.
For simplicity, we first recall an interpretation of K-semistability of S. Paul in terms of polytopes (cf. \cite{paul08}).
\begin{prop}\label{prop:paul-1}
Let $(v,w)$, $\bf V$, $\bf W$, $\bG$ and $\bf T$ as in Section 1. Then $(v,w)$ is K-semistable if and only if for any $\tau\in \bf G$,
we have
\begin{equation}\label{eq:sp-1}
\mathcal{N}(\tau (v))\subset \mathcal{N}(\tau(w)).
\end{equation}
\end{prop}

Let us explain why this proposition holds. We observe that for any one-parameter subgroup $\lambda_\ell\subset \bf T$ ($\ell\in \bf N_{\mathbb Z}$):
\begin{eqnarray}
w_{\lambda_\ell}(\tau(w))&=&\inf\{\, (\ell,m)\,|\,m\in A(\tau(w))\,\},\label{eq:sp-2}\\
w_{\lambda_\ell}(\tau(v))&=&\inf\{\, (\ell,m)\,|\,m\in A(\tau(v))\,\}.\label{eq:sp-3}
\end{eqnarray}
If \eqref{eq:sp-1} is false, then we can find $m'\in A(\tau(v))$ which is not contained in $\mathcal{N}(\tau(w))$. This means that there is a linear function $\ell:
\bM_\RR\mapsto \RR$ such that $\ell |_{A(\tau(w))} \ge 0$ and $\ell(m') <0$. Since both $\bV$ and $\bW$ are rational representations, we may take $\ell$ to be integral, i.e., $\ell\in \bN_\ZZ$. Thus,
$$ w_{\lambda_\ell}(\tau(v))\,\le\,(\ell,m')\,<\,0 ~~~{\rm while}~~~ w_{\lambda_\ell}(\tau(w)) \,\ge\,0.$$
This implies that $(v,w)$ is not K-semistable. So the K-semistability implies \eqref{eq:sp-1}. The other direction is clear.

Similarly, we can express the K-stability of $(v,w)$ in terms of polytopes $\mathcal{N}(\tau (v))$, $\mathcal{N}(\tau (w))$ and $\mathcal{N}(\tau)$ corresponding to representations $\bf V$, $\bf W$ and $\bf {gl}|$. Note that For any $\tau\in \bf G$, $\mathcal{N}(\tau)$ is always the standard N-simplex $\mathcal{N}(I)$.
In view of \eqref{eq:sp-2} and \eqref{eq:sp-3}, we see that $(v,w)$ is K-stable if and only if it is K-semistable and
for any $\tau\in \bG$ and $\ell \in \bf N_{\mathbb Z}$,
\begin{eqnarray}\label{eq:sp-4}
&&q\,\inf\{(\ell,m)\,|\,m\in \mathcal{N}(I) \}< \inf\{(\ell,m)\,|\,m\in A(\tau(v))\}.\nonumber\\
&\Rightarrow& \inf\{(\ell,m)\,|\,m\in A(\tau(w))\} < \inf\{(\ell,m)\,|\,m\in A(\tau(v))\}
\end{eqnarray}
Theorem \ref{th:HMP} is equivalent to that for any there is an integer $k$ such that $\tau\in \bG$ and $\ell \in \bN_\ZZ$,
\begin{eqnarray}\label{eq:sp-5}
&&k\,\left(\inf\{(\ell,m)\,|\,m\in A(\tau(v))\}\,-\,\inf\{(\ell,m)\,|\,m\in A(\tau(w))\}\right)\\
&\ge& \inf\{(\ell,m)\,|\,m\in A(\tau(v))\}\,- \,q\,\inf\{(\ell,m)\,|\,m\in \mathcal{N}(I)\}.\nonumber
\end{eqnarray}
By the semi-continuity, there should be only finitely many possibly different $\mathcal{N}(\tau (v))$ and $\mathcal{N}(\tau (w))$,
so we believe that there will be a direct proof of \eqref{eq:sp-5}.

This above approach to proving \eqref{eq:sp-5} may be applied to proving uniform K-stability for a K-stable polarized manifold $(M, L)$ as in Section 5. For each sufficiently large $\ell$, we have a pair $(R_\ell(M),\Delta_\ell(M))$ associated to the embedding given by $H^0(M,L^\ell)$. As we have shown in Section 5,  such a pair $(R_\ell(M),\Delta_\ell(M))$ is uniformly K-stable. Since the ring $R(M,L)$ is finitely generated,
where
$$R(M,L)\,=\,\oplus_{\ell=1}^\infty  H^0(M,L^\ell),$$
the related polytopes $\mathcal{N}(R_\ell(M)) $, $\mathcal{N}(\Delta_\ell(M))$ for all $\ell$ should be determined by finitely many such polytopes, so we may have an uniform estimate for $k$ appeared in \eqref{eq:sp-5} for all $\ell$. This would lead to a proof of the uniform K-stability of $(M,L)$.
However, because the relations among those polytopes $\mathcal{N}(R_\ell(M))$, $\mathcal{N}(\Delta_\ell(M))$ are implicit, the proof may be tricky.

\section{Appendix A: Proof of Lemma \ref{lemm:richardson}}

In this appendix, following \cite{paul13}, we give a proof of Lemma \ref{lemm:richardson} which we restate as follows:

\begin{lemm}\label{lemm:richardson'}
For any $z\in \tilde {\bf T}\backslash \bf T x$, where $x=([v,w],[v,u])$, there is a one-parameter subgroup $\lambda: \mathbb C^* \mapsto \bf T$ such that $\lambda(t) x \rightarrow \tau(z)$ for some $\tau\in \bf T$ as $t \rightarrow 0$.
\end{lemm}

\begin{proof} We may assume that $\tilde {\bf T}$ is a smooth subvariety in a projective space ${\mathbb{P}}(\bf E)$ for a $\bf G$-representation $\bf E$. We can have the following decomposition:
$$\bf E\,=\, \sum_{a\in A} \,\bf E_a,~~~{\rm where}~~{\bf E}_{a}\,=\,\{\,u\in \bf E \, |\, t\cdot u\,=\,a(t)\, u \ , \, t \in \bf T\,\}.$$
Accordingly, we can write
$$z\,=\,\sum_{a\in A}\,z_a,~~~z_a\in \bf E_a.$$
Pick up a $z_a\not= 0$, say the first one $z_{a_0}$, then we have $\bf E\,=\,\CC\oplus \bf F$ and can regard $\bF$ as a subspace in ${\mathbb{P}}(\bf E)$, moreover,
$\bf F$ admits $\bf T$-action and a decomposition:
\begin{equation}\label{eq:appen-1}
\bf F\,=\, \sum_{a\in A'} \,\bf F_a,~~~{\rm where}~~{\bf F}_{a}\,=\,\{\,u\in \bF \, |\, t\cdot u\,=\,a(t)\, u \ , \, t \in \bf T\,\},
\end{equation}
where  $ a'\,=\,A\backslash \{a_0\}$. We fix a basis $\{\bf e_i\}_{1\le i\le d'}$ of $\bf F$ such that $ t \cdot \bf e_i = a_i (t) \bf e_i$ for
some $a_i\in A'$.\footnote{For $i\not= j$, we may still have $a_i\,=\,a_j$.}
Since $z\in \bf F\subset {\mathbb{P}}(\bf E)$, $\bf T x\subset \bf F$. Clearly, $Z\,=\,\tilde {\bf T}\cap \bf F \backslash \bf T x$ is $\bf T$-invariant and closed, furthermore,  there are $t_\ell\in \bf T$ such that $t_\ell\rightarrow 0$ and $t_\ell\cdot x$ converge
to $z \in Z$ as $\ell$ goes to $\infty$.

By rearranging the indices, we may write
$$x\,=\,\sum_{i=1}^{d'} x_i \,\bf e_i,~~~z\,=\,\sum_{j=k}^{d'} z_j \,\bf e_j ,$$
where $1\le k \le d'$, $x_i\not=0$ and $z_j\not=0$.
Our assumption implies that $a_i (t_\ell) x_i$ converge to $0$ for $i < k$ and converge to $z_i $ for $i \ge k$.
Since $x_i\not= 0$ and $z_j\not=0$, we get
$$\lim_{\ell\to \infty}  a_i (t_\ell)\,=\,0~~\forall \,i \,<\,k~~{\rm and}~~~\lim_{\ell\to \infty}  a_j(t_\ell)\,=\,\frac{z_j}{x_j}\,=\,1~~\forall\, j\,\ge \,k.$$

Consider the quotient
$$ \pi:  \bM_\RR\,\mapsto\, W\,= \,\bM_\RR / \bM_z,~~~{\rm where}~\bM_z \,=\,\bigoplus_{j=k}^{d'} \,\RR\cdot a_j.$$
Denote by $\Delta$ the convex hull (in $W$) of $a_i$ for $i=1,\cdots, k-1$.

We claim that $0\notin \Delta$. This can be shown as follows: If the claim is false,
then there are real constants $r_1,\cdots, r_{ k-1}\,\ge\,0$ such that
$${\rm some}~r_{i}\,>\, 0,~~~ \sum_{i=1}^{k-1} r_i\,a_i \,=\,0~~{\rm mod}~\bM_z,$$
hence, there are $c_k,\cdots,c_{d'}$ such that
$$
\sum_{i=1}^{k-1} r_i\,a_i\,=\,\sum_{j=k}^{d'} c_j\,a_j .
$$
Hence, for all $t\in \bf T$, we have
\begin{equation}\label{product}
\prod_{i=1}^{k-1} \, |a_i(t)|^{r_i}\,=\,\prod_{j=k}^{d'} |a_j(t)|^{c_j}.
\end{equation}
By plugging in the sequence $\{t_\ell\}$, we get a contradiction since the left side of (\ref{product})
tends to zero while the right side does not. This proves our claim.

Using this claim and the Hyperplane Separation Theorem, one can get a linear functional $f: W \mapsto \RR$ such that
$f(\pi(a_i))\,>\,0$ for $i < k$.
Furthermore, one can choose this to be \emph{rational}.
Next we lift $f$ to $$F\,=\,f\cdot \pi: \bM_\RR \mapsto \RR.$$
Then $F$ is a \emph{rational} linear functional on $\bf M_\RR$. Multiplying it by an integer, we may even assume that $F$ is integer-valued.
Therefore, it induces an one-parameter subgroup $\lambda: \mathbb C^*\mapsto \bf T$ satisfying:
\begin{equation}
\label{eq:lambda}
\lim_{t \to 0}\,\lambda (t) \,x\,=\, \sum_{j=k}^{d'} x_j \bf e_j \,=\,x'.
\end{equation}
Since $a(t_\ell) x_j$ converge to $z_j$ for any $j\ge k$, we have $ \overline{\bf T z} \subset \overline{\bf T x'}$. On the other hand, both
orbits $\bf T z$ and $\bf T x'$ have the same dimension, so $\bf T z\, =\, \bf T x'$. It follows that there is a $\tau\in \bf T$ such that
\begin{equation}\label{eq:lambda-2}
\lim_{t \to 0}\,\lambda (t) \,x\,=\, \tau(z).
\end{equation}
The lemma is proved.

\end{proof}

\vskip5mm

\section {Appendix  B:  A variant proof of Theorem \ref{th:HMP} by Li Yan and Xiaohua Zhu }
In this appendix,   we give an element  proof of Theorem \ref{th:HMP}   by   an approach in Section 6.

\subsection{Stability and polytopes of weights }

Let   $\bf V$ be a  linear space  with a $r$-dimensional   torus $T$-action.  Then we  can decompose it into a  direct sum of $\bf T$-invariant subspaces such that
\begin{eqnarray}
V=\oplus_{\mu\in I_V}m_\mu\check V_\mu,
\end{eqnarray}
where $I_V$ is a finite set of characters of $\bf T$,
$V_\mu$'s are   one-dimensional irreducible representations of $\bf T$ with    multiplicity $m_\mu$. Thus  for any $v\in \bf V$,  there are some $\alpha\in I_V$ and
non zero constants $c_\alpha^v\in\mathbb C$ such that
$$v=\sum_\alpha c_\alpha^ve_\alpha, $$
where  $e_\alpha\in m_\alpha\check V_\alpha$.    It follows that for any one parameter sub-group   $\lambda(t)$ of $\bf T$ with its Lie algebra $\lambda\in Z^r$,
$$\lambda(t)(v)=\sum_\alpha c_\alpha^vt^{\langle\lambda,  \Lambda_\alpha\rangle }e_\alpha,$$
where $  \Lambda_\alpha$ are weights  as   characters  of $\bf T$ associated to $v$.

Set a  convex hull of weights $\{   \Lambda_\alpha\}$  by     $\mathcal N (v)$  and   Let  $v_{\mathcal N(v)}(\cdot)$ be  a support function of $\mathcal N (v)$.
 Then  by (\ref{eq:weight-2}),   it  is easy to see that     the weight  $w_\lambda(v)$ of  $v$ associated to    $\lambda(t)$   is equal to
\begin{eqnarray}\label{weight-support}
w_\lambda(v)=\min_{\alpha}  \langle\lambda, \Lambda_\alpha\rangle= -v_{\mathcal N(v)}(-\lambda).
\end{eqnarray}

Now we  assume that  $V$ is  a  linear  presentation space  of a  reductive Lie group $\bf G$ with a maximal    torus $\bf T$.  Then  for any one parameter subgroup $\lambda(t)'$, there is a $\sigma\in \bf G$ and   one parameter sub-group   $\lambda(t)$ of $\bf T$  such that
$$\lambda(t)' = \sigma\cdot  \lambda(t)  \cdot \sigma^{-1}\subseteq  \bf T'=\sigma\cdot \bf T\cdot \sigma^{-1}.$$
Thus  $\sigma(\check V_\mu)$   are   one-dimensional irreducible representations of torus  $\bf T'$ with property
$$\mathbf {V}=\oplus_{\mu\in I_V} m_\mu \sigma(\check V_\mu).$$
In fact,  for the vector $v'=\sigma^{-1}v$, there are non zero constants $ c_{\alpha'}^{v'}\in\mathbb C$ such that
$$v'=\sum_{\alpha'} c_{\alpha'}^{v'}e_{\alpha'}$$
and
$$\lambda(t)' v'= \sum_{\alpha'} c_{\alpha'}^{v'} t^{\langle\lambda,  \Lambda_{\alpha'} \rangle } e_{\alpha'},$$
where  $e_{\alpha'}\in m_{\alpha'}  \sigma(\check V_{\alpha'})$ and   $  \Lambda_{\alpha'}$ are weights  as   characters  of $\bf T$ associated to $v'$.
Denote by  the convex hull of weights $\{   \Lambda_{\alpha'}\}$  by  $\mathcal N^\sigma (v)$.  It is clear that such  convex hulls $ \mathcal N^\sigma (v)$ are finitely  many  since $\bf V$ is a finitely  dimensional   linear space. Moreover,  the weight  $w_{\lambda'}(v)$ of    $\lambda(t)'$ is given by
\begin{eqnarray}\label{weight-support-2}
w_{\lambda'}(v)=\min_{\alpha'}  \langle\lambda, \Lambda_{\alpha'}\rangle=-v_{\mathcal N^\sigma(v)}(-\lambda).
\end{eqnarray}
Hence, by \eqref{weight-support-2},   we have the following proposition.

\begin{prop}\label{polytope-stability}
\begin{itemize}
\item[(1)]  Pair $(v,w)$ is K-semistable with respect to  $\bf G$  iff
 \begin{eqnarray}\label{semistable-polytope}\mathcal N^\sigma(v)\subseteq\mathcal N^\sigma(w),~\forall ~\sigma\in \bf G.
 \end{eqnarray}
\item[(2)]  Pair $(v,w)$ is K-stable  with respect to  $\bf G$  iff   (\eqref{semistable-polytope}) is satisfied  and
\begin{eqnarray}\label{stable-polytope}\{x|~v_{\mathcal N^\sigma(v)}(x)=v_{\mathcal N^\sigma(w)}(x)\}\subseteq\{x|v_{\mathcal N^\sigma(v)}(x)=v_{\deg(V)\mathcal N^\sigma( \bf I)}(x)\},~\forall ~\sigma\in \bf G.
\end{eqnarray}
\item[(3)] Pair $(v,w)$ is uniform  K-stable   with respect to  $\bf G$  iff  there is an $m\in\mathbb N_+$ such that
\begin{eqnarray}\label{uniformstable-polytope}
\left(1-\frac1m\right)\mathcal N^\sigma(v)+\frac1m\deg(\bf V)\mathcal N^\sigma({\bf I})\subseteq\mathcal N^\sigma(w).
\end{eqnarray}
\end{itemize}
\end{prop}

\subsection{$K$-stability  implies $ K$-uniform stability}. As  we discussed  Section 8.1,     the numbers of convex  hulls $\mathcal N^\sigma (v)$,  $ \mathcal N^\sigma (w)$  and    $ \mathcal N^\sigma ({\rm Id})$.  are  all  finite.    Thus by Proposition \ref{polytope-stability} (3),  Theorem   \ref{th:HMP}  is reduced to prove

\begin{theo}\label{unistthm}
Let $\bf V,\bf W$ be two linear space with  a torus $T$-action. Suppose that  pair $(v,w)\in (\bf V\setminus\{0\})\times (\bf W\setminus\{0\})$ is K-stable    with respect to  $\bf T$.  Then   there is an $m\in\mathbb N_+$ such that
\begin{eqnarray}\label{uniformstable-polytope-1}
\left(1-\frac1m\right)\mathcal N(v)+\frac1m\deg(V)\mathcal N({\bf I})\subseteq\mathcal N(w).
\end{eqnarray}

\end{theo}

To prove Theorem  \ref{unistthm},  we first state a result for the support function of convex polytope.
Let
\begin{eqnarray}\label{polytopedef}
P=\cap_{A\in\mathcal A}   \{l_A(y)=a_A-u_A(y)\geq0\}
\end{eqnarray}
be a convex polytope and $F^P_A\subseteq \{l_A(y)=0\}$ be its (codim 1) facets. For $\mathcal I\subseteq \mathcal A$,  we denote
$$F^P_\mathcal I=\cap_{i\in\mathcal I}F^P_i.$$

\begin{lemm}\label{support-function-value}
The set
\begin{eqnarray*}
\Omega_\mathcal I^P=\{x|~v_P(x)=\langle x,y\rangle,~\forall y\in F^P_\mathcal I\}=\text{Span}_{\mathbb R_{\geq0}}\{u_i|i\in\mathcal I\}.
\end{eqnarray*}
\end{lemm}

\begin{proof}
By definition
\begin{eqnarray*}
\begin{aligned}
\Omega_\mathcal I^P&=\{x|~\langle x,y-y'\rangle\geq0,~\forall y\in F^P_\mathcal I\text{ and }y'\in P\}\\
&=\cap_{y\in F^P_\mathcal I,y'\in P}\left(\text{Span}_{\mathbb R_+}\{y-y'\}\right)^\vee.
\end{aligned}
\end{eqnarray*}
By the  convexity of $P$, this is equivalent to
\begin{eqnarray}\label{dual cone}
\begin{aligned}
\Omega_\mathcal I^P=\left(\text{Span}_{\mathbb R_+}\{y-y'|y\in F^P_\mathcal I\text{ and }y'\in \cup_{i\in\mathcal I}F^P_i\}\right)^\vee.
\end{aligned}
\end{eqnarray}
Thus  the lemma is true.
\end{proof}

\begin{proof}[Proof of Theorem \ref{unistthm}]
Since $(v,w)$ is K-stable, by Proposition \ref{polytope-stability} (1), we have \eqref{semistable-polytope}. In particular,  if
$$\text{dist}(\partial\mathcal N(v),\partial\mathcal N(w))\geq\epsilon_0>0,$$
 there is a $\delta_0>0$  such that for any $\delta\in(0,\delta_0)$,
$$\left(1-\delta\right)\mathcal N(v)+\delta \deg(V)\mathcal N({\bf  I}) \subseteq\mathcal N(w),$$
since $\mathcal N(\bf I)$ is compact. The theorem then follows from Proposition \ref{polytope-stability} (3).

In the following,  we  assume that
$$\text{dist}(\partial\mathcal N(v),\partial\mathcal N(w))=0.$$
Then
\begin{eqnarray}\label{case-2}\begin{aligned}&\{x|~v_{\mathcal N(w)}(x)=v_{\mathcal N(w)}(x)\}\\&=\{x|~v_{\mathcal N(w)}(x)=\langle x,y\rangle\text{ and } v_{\mathcal N(w)}(x)=\langle x,y\rangle,~\text{for some } y\in\partial \mathcal N(v)\cap\partial \mathcal N(w)\}.\end{aligned}\end{eqnarray}
By Proposition \ref{polytope-stability} (2), we see that    for any $x$ as above there is  $y\in \partial (\deg(V) \mathcal N({ \bf I})) \cap\mathcal N(v)\cap\partial \mathcal N(w)  $  such that
$$v_{ \mathcal N(v)}(x)=v_{\deg(V) \mathcal N({ \bf I})}(x)=\langle x,y\rangle.$$
In particular, 
\begin{align}\label{three-intersection} \left(\partial \mathcal N(v)\cap\partial\deg(V) \mathcal N({ \bf I})\cap\partial \mathcal N(w)\right)\not=\emptyset.
\end{align}

Recall by the definition of $\deg(V)$ that
$$\mathcal N(v)\subseteq\deg(V)\mathcal N({\bf  I}).$$
Then,  by (\ref{three-intersection}),  there are some facets $F_\mathcal I^{\mathcal N(v) }, F_\mathcal J^{\deg(V)\mathcal N({\rm Id}) }$ and $F_\mathcal K^{\mathcal N(w) }$ of $\mathcal N(v), \deg(V)\mathcal N({\bf  I})$ and ${\mathcal N(w) }$, respectively, such that
$$F_\mathcal I^{\mathcal N(v)}\subseteq\left( F_\mathcal J^{\deg(V)\mathcal N({\rm Id})}\cap F_\mathcal K^{\mathcal N(w)}\right), $$
 where each $F_\spadesuit^\diamondsuit$ above contains a point $y\in\left(\partial \mathcal N(v)\cap\partial\deg(V) \mathcal N({ \bf I})\cap\partial \mathcal N(w)\right)$ in its relative interior.
Thus by  Proposition \ref{polytope-stability} (2)  together with Lemma \ref{support-function-value} and \eqref{case-2},   we get
 (see Figure 1-2 for 4 possible stable cases in dimension $2$, according to different codimensions of $F_\spadesuit^\diamondsuit$),
$$\Omega_\mathcal K^{\mathcal N(w) }\subseteq \Omega_\mathcal J^{\deg(V)\mathcal N({\bf I}) }\subseteq \Omega_\mathcal I^{\mathcal N(v) }.$$
\begin{figure}[h]
\centering
\includegraphics[height=2in]{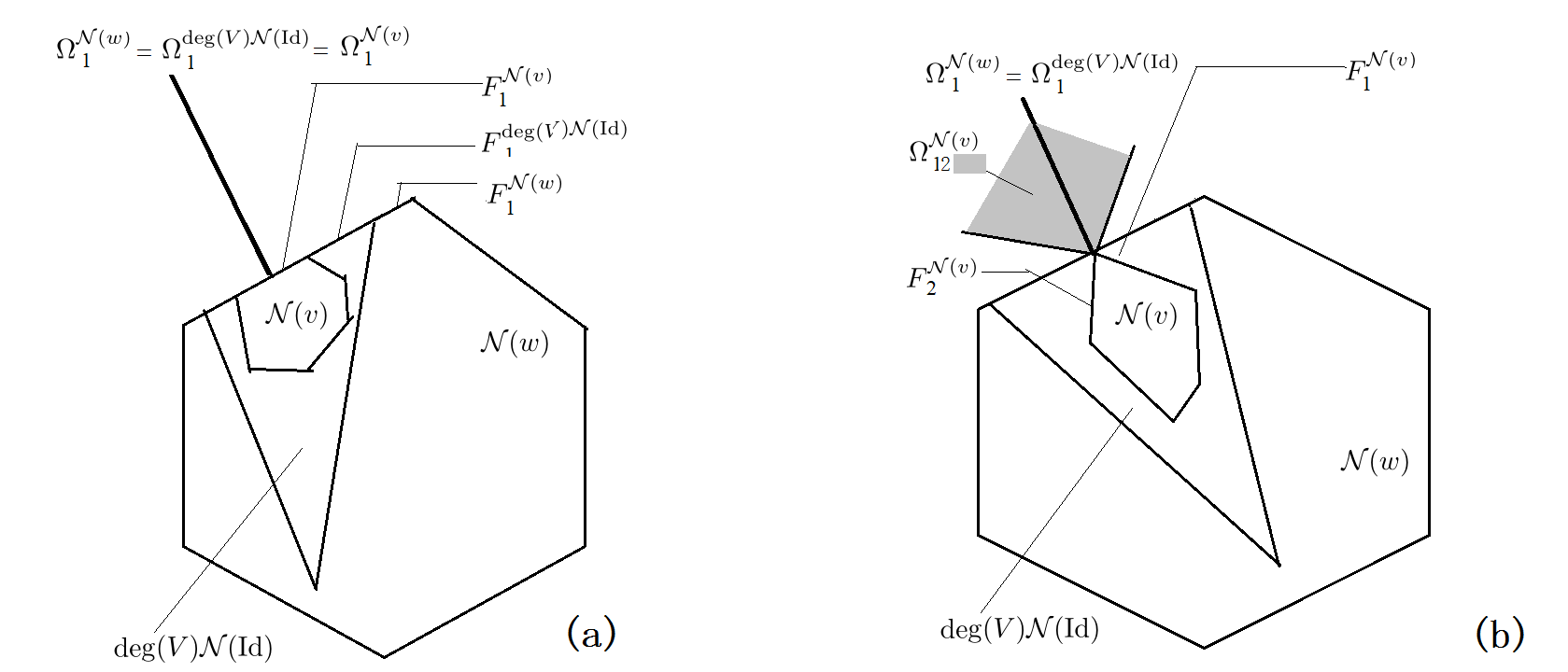}
\caption{Stable examples in $2$-dimension}
\end{figure}
\begin{figure}[h]
\centering
\includegraphics[height=2in]{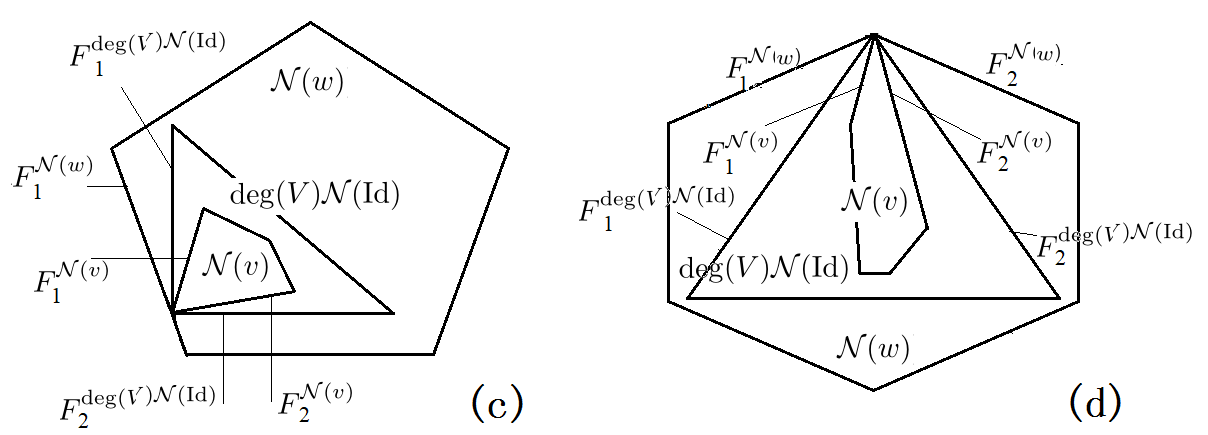}
\caption{Stable examples in $2$-dimension}
\end{figure}
Hence, by the duality property \eqref{dual cone},  there are compact sets $\Omega_1\subseteq\Omega_2$, both contain $F_\mathcal I^{\mathcal N(v) }$ such that (for simplicity, see Figure 3 for the case (d) in Figure 2),
\begin{eqnarray}\label{cones-local}(\Omega_i\cap {\mathcal N(v) })\subseteq(\Omega_i\cap {\deg(V)\mathcal N({\bf  I}) })\subseteq(\Omega_i\cap  {\mathcal N(w) }),~i=1,2.\end{eqnarray}
\begin{figure}[h]
\centering
\includegraphics[height=2in]{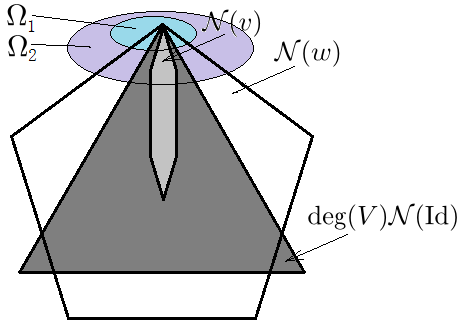}
\caption{}
\end{figure}

Without loss of generality, we may assume there is only one facet $F_\mathcal I^{\mathcal N(v) }$ of $\mathcal N(v)$ such that
\begin{eqnarray}\label{local}
F_\mathcal I^{\mathcal N(v) }\subseteq\left(\partial \mathcal N(v)\cap\partial\deg(V) \mathcal N({\bf  I})\cap\partial \mathcal N(w)\right).
\end{eqnarray}
We consider the following two cases.

\item[(i)] Outside $\Omega_1$. Then we have
$$\text{dist}( \mathcal N(v)\setminus\Omega_1,\partial \mathcal N(w))\geq\epsilon_0>0.$$
Thus  there is a $\delta'_0>0$  such that for any $\delta'\in(0,\delta'_0)$,
$$\left(1-\delta'\right)( \mathcal N(v) \setminus\Omega_1)+\delta' \deg(V)\mathcal N({\bf  I}) \subseteq \mathcal N(w).$$

\item[(ii)] Inside  $\Omega_1$.    Then by \eqref{cones-local}, \eqref{local} and convexity of each $\mathcal N(\cdot)$, there is a $\delta_0>0$  such that for any $\delta\in(0,\delta_0)$,
$$\left(1-\delta\right)(\Omega_1\cap \mathcal N(v) )+\delta \deg(V)\mathcal N( \text{Id}) \subseteq(\Omega_2\cap \mathcal N(w) ).$$
Therefore,  by choosing $$m=\left[\frac1{\min\{\delta_0,\delta'_0\}}\right]+1,$$
we get  (\ref{uniformstable-polytope-1}) immediately.

\end{proof}

\subsection{{\rm CM}-stability} In this subsection, we first give a direct proof of Theorem \ref{th:HMP2}, then we use   Theorem   \ref{th:HMP} to derive Theorem \ref{th:properness}.

Recall the  functional  on $\bf G$ associated to the pair $(v,w)\in \bf V\setminus \{0\} \times \bf W\setminus \{0\}$ in Section 2,
$$p_{v,w}(\sigma)\,=\,\log||\sigma (w)||^2\,-\,\log||\sigma(v)||^2, ~\sigma\in \bf G.$$
 Without loss of generality, we may assume that  $\|\cdot\|$ is an euclidean norm of vector in $\bf V$ or $\bf W$ and  it is invariant under the maximal compact group  $\bf K$ which complexifies $\bf G$.  We note that there is a uniform constant $\delta_0$
 such that for any indices $i,l$ it holds
$$\inf\{ |(k v)_i| ~(k v)_i\neq 0, ~ k\in K\}\ge \delta_0~{\rm and}~\inf\{ |(k w)_l|~(k w)_l\neq 0, k\in K\}\ge \delta_0.$$
Thus  for any one parameter subgroup $\lambda(s)$, and $k,k'.  \in \bf K$, we get
\begin{align}&   p_{v;w}(k'\lambda(s) k) =p_{v;w}( k^{-1}\lambda(s)k) =p_{v;w}(\lambda(s)k)\notag\\
&=-[{\rm max}_{a\in \Lambda(kv)} \langle a, \mathfrak t_\lambda\rangle -{\rm max}_{a'\in \Lambda(kw)}   \langle a', \mathfrak t_\lambda\rangle ]  {\rm log}\frac{1}{s} +\mathcal O(1)\notag\\
&=  (w_{k^{-1}\lambda k}(v)-w_{ k^{-1}\lambda k}(w)){\rm log} \frac{1}{s} +\mathcal O(1) ,~\forall ~|s|\ll1, \notag
\end{align}
where $\mathcal O(1)$ means a uniform bounded constant and  $\mathfrak t_\lambda$ denote the Lie algebra of $\lambda(s)$.

By  the semi-stability of $(v,w)$,     the function
$$f_{kv, kw} ( \mathfrak t)= [ {\rm max}_{a'\in \Lambda(kw)}  \langle a', \mathfrak t\rangle - {\rm max}_{a\in \Lambda(kv)} \langle a, \mathfrak t\rangle]\ge 0$$
for all $  \mathfrak t\in \mathbb Z^n$ and so  for all $ \mathfrak t\in \mathbb Q^n$.  Thus,
$f_{kv, kw} ( \mathfrak t)$  can be extended to a non-negative continuous function in $\mathbb R^n$.

Fix a small $s>0$, for example,  $s=\frac{1}{100}$.  Then,  for  any  $t=(t_1,...,t_n)\in \bf T$, there is $ \mathfrak t=({\mathfrak t}_1,...,{\mathfrak t}_n)\in \mathbb R^n$ such that $$(|t_1|,...,|t_n|) =(s^{{\mathfrak t}_1},..., s^{{\mathfrak t}_n}).$$
It follows that
\begin{align}\label{p-energy} p_{v;w}(k'  t k)=p_{v;w}(t k)&=f_{kv, kw} ( \mathfrak t){\rm log} \frac{1}{s} +\mathcal O(1)\notag\\
&\ge -C.
\end{align}
Hence,  we prove

\begin{theo}
\begin{itemize}
\item[(1)]    Pair $(v,w)$ is  $K$-semistable with  respect to  $G$. Then
\begin{align}\label{lower-bound} p_{v;w}(\sigma)\ge -C, ~\forall ~\sigma\in \bf G. 
\end{align}
\item[(2)]  Pair $(v,w)$ is   $K$-stable   with respect to  $\bf G$. Then    there is an $m\in\mathbb N_+$ such that
\begin{align}\label{properness-p} m p_{v;w}(\sigma )&\ge
 p_{v;{\bf  I}}(\sigma) -C\notag\\
&\ge \deg(V)\log ||\sigma||^2- \log ||\sigma (v)||^2-C,  ~\forall ~\sigma\in \bf G.
\end{align}

\end{itemize}

\end{theo}

\begin{proof} By the $\bf K\times\bf K$ decomposition of $\bf G$, for any $\sigma\in \bf G$, there are $k,k'  \in \bf K$ and $t\in \bf T$  such that
$\sigma=k'  t k$. Thus (\ref{lower-bound}) follows from (\ref{p-energy}). Next,  we can consider the following  energy $\tilde p_{v;w;{\bf  I}}$,
$$\tilde p_{v;w;{\bf  I}}(\sigma)=  m p_{v;w}(\sigma)-   p_{v;{\bf I}}(\sigma),~\forall ~\sigma\in \bf G.$$
Then as in the proof of (\ref{p-energy}), we have
$$\tilde p_{v;w;{\bf  I}}(k'  t k)=\tilde f_{kv, kw, k{\bf  I}} ( \mathfrak t){\rm log} \frac{1}{s} +\mathcal O(1),
$$
where 
\begin{align} &\tilde f_{kv, kw, k{\bf  I}} ( \mathfrak t)\notag\\
&= m[ {\rm max}_{a'\in \Lambda(kw)}  \langle a', \mathfrak t\rangle - {\rm max}_{a\in \Lambda(kv)}  \langle a, \mathfrak t\rangle]-[{\rm max}_{a''\in \Lambda(k{\bf  I})}  \langle a'', \mathfrak t\rangle - {\rm max}_{a\in \Lambda(kv)}  \langle a, \mathfrak t\rangle]
\notag
\end{align}
is a non-negative continuous function in $\mathbb R^n$ by Theorem  \ref{th:HMP}. Thus we get
$$\tilde p_{v;w;{\bf I}}(\sigma)\ge -C,~\forall ~\sigma\in \bf G,$$
which implies (\ref {properness-p}).

\end{proof}


\begin{thebibliography}{plain}

\frenchspacing

\bibitem[BHJ17]{bhj16}
Boucksom, S., Hasamoto, T. and Jonsson, M.: Uniform K-stability and asymptotics of energy functionals in K\"ahler geometry. To appear in JEMS.

\bibitem[MFK94]{mumford} Mumford, D., Forgarty, J. and Kirwan, F.: Geometric Invariant Theory, Ergebnisse der Mathematik und ihrer Grenzgebiete, Vol. 34, Springer-Verlag,  Berlin, 1994.

\bibitem[Od13]{odaka} Odaka, Y.: On the moduli of K\"ahler-Einstein Fano manifolds, Proceeding of Kinosaki algebraic geometry symposium 2013, arXiv:1211.4833.


\bibitem[Pa08]{paul08} Paul, S.: Hyperdiscriminant polytopes, Chow polytopes, and Mabuchi energy asymptotics. Ann. Math., 175 (2012), 255-296.

\bibitem[Pa12a]{paul12a} Paul, S.: CM stability of projective varieties. Preprint, arXiv:1206.4923.

\bibitem[Pa12b]{paul12b} Paul, S.: A numerical criterion for K-energy maps of algebraic manifolds. Preprint, arXiv:1210.0924v1.

\bibitem[Pa13]{paul13} Paul, S.: Stable pairs and coercive estimates for the Mabuchi functional. Preprint, arXiv:1308.4377.

\bibitem[PT04]{paultian04} Paul, S. and Tian, G.: Analysis of geometric stability. Int. Math. Res. Notices. IMRN 2004, no. 48, 2555-2591.

\bibitem[PT06]{paultian} Paul, S. and Tian, G.: CM stability and the generalized Futaki invariant II.  Ast\'erisque No. 328 (2009), 339-354.

\bibitem[Ti97]{tian97} Tian, G: K\"ahler-Einstein metrics with positive scalar curvature. Invent. Math., 130 (1997), 1-39.

\bibitem[Ti12]{tian12} Tian, G.: K-stability and K\"ahler-Einstein metrics. Comm. Pure Appl. Math. 68(2015), no. 7, 1085-1156.

\bibitem[Ti13]{tian13-2} Tian, G.: K\"ahler-Einstein metrics on Fano manifolds.  Jap.  J. Math. 10 (2015), 1-41.

\bibitem[Ti14]{tian14} Tian, G.: K-stability implies CM-stability. Preprint, arXiv:1409.7836.

\end{thebibliography}
\end{document}